\documentclass[12pt]{article}
\usepackage{amssymb,amsmath}
\usepackage{latexsym,bm}

\usepackage{amsfonts}
\usepackage{latexsym,amsmath,amssymb,amsfonts,epsfig,psfrag,url,graphics,ifpdf,multicol}
\usepackage{eepic,color,colordvi,amscd,amsthm}
\usepackage[section]{algorithm}
\usepackage{algpseudocode}
\usepackage[numbers,sort&compress]{natbib}
\usepackage{indentfirst,graphics,epsfig}
\usepackage{graphicx}
\usepackage{graphics}
\usepackage{tikz}
\usepackage{booktabs}
\usepackage{subfigure}

\newtheorem{theo}{Theorem}
\newtheorem{lem}[theo]{Lemma}

 \theoremstyle{definition}

\theoremstyle{remark}

\newcounter{casenum}[theo]

\newcounter{subcasenum}[theo]

\newcounter{claimnum}[theo]

\allowdisplaybreaks
\usepackage{indentfirst,subfig}

\begin{document}
\thispagestyle{plain}

\begin{center} {\Large Estimating the circumference of a graph in terms of its leaf number
}
\end{center}
\pagestyle{plain}
\begin{center}
{
  {\small  Jingru Yan \footnote{ E-mail address: mathyjr@163.com}}\\[3mm]
  {\small  Department of Mathematics, East China Normal University, Shanghai 200241,  China }\\

}
\end{center}

\begin{center}

\begin{minipage}{140mm}
\begin{center}
{\bf Abstract}
\end{center}
{\small   Let $\mathcal{T}$ be the set of spanning trees of $G$ and let $L(T)$ be the number of leaves in a tree $T$. The leaf number $L(G)$ of $G$ is defined as $L(G)=\max\{L(T)|T\in \mathcal{T}\}$. Let $G$ be a connected graph of order $n$ and minimum degree $\delta$ such that $L(G)\leq 2\delta-1$. We show that the circumference of $G$ is at least $n-1$, and that if $G$ is regular then $G$ is hamiltonian.

{\bf Keywords.} Leaf number, circumference, hamiltonian}

{\bf Mathematics Subject Classification.} 05C38, 05C45
\end{minipage}
\end{center}

\section{Introduction}
We will deal with only finite nontrivial simple graphs. Let $G$ be a graph with vertex set $V(G)$ and edge set $E(G)$. The order and size of a graph $G$ are its number of vertices and edges, respectively. The notations $N_G(v)$ and $N_G[v]$ denote the neighborhood and closed neighborhood of $v\in V(G)$, respectively. The degree of $v$ is $d_G(v)=|N_G(v)|$. $\delta(G)$ and $\Delta(G)$ denote the minimum and maximum degree of a graph $G$, respectively. If the graph $G$ is clear from the context, we will omit it as subscript. For terminology and notations not explicitly described in this paper, the reader is referred the books \cite{BM,W}.

Let $\mathcal{T}$ be the set of spanning trees of $G$. $L(T)$ denotes the number of leaves in a tree $T$, where a leaf means a vertex of degree 1. Then the leaf number $L(G)=\max\{L(T)|T\in \mathcal{T}\}$. Many researchers have estimated the circumference of graphs by various invariants. The purpose of this paper is to estimate the circumference of a connected graph $G$ by the two invariants $\delta(G)$ and $L(G)$.

DeLaVi\~{n}a's computer program, Graffiti.pc, posed attractive conjectures \cite{DE} and some of the conjectures speculate sufficient conditions for traceability based on the minimum degree and leaf number. In 2013, Mukwembi gave a partial solution to the Graffiti.pc 190a. He \cite{M12013} showed that if $G$ is a finite connected graph with minimum degree $\delta(G)\geq 5$, and leaf number $L(G)$ such that $\delta(G)\geq L(G)-1$, then $G$ is hamiltonian and thus traceable. In the same year, he \cite{M22013} relaxed the condition $\delta(G)\geq 5$ to $\delta(G)\geq 3$. After that, Mukwembi \cite{M32013} proved that if $G$ is a connected claw-free graph with $\delta(G)\geq (L(G)+1)/2$, then $G$ is hamiltonian. In recent years, several authors reported on sufficient conditions for a graph to be hamiltonian or traceable based on minimum degree and leaf number, see \cite{M,MM,MMM,MMMV,MM2}.

We state the following results, some of which will be used later in this paper.

\begin{theo}\label{th1}\cite{MMMV}
If $G$ is a connected graph with $\delta(G)\geq (L(G)+2)/2$, then $G$ is hamiltonian.
\end{theo}

\begin{theo}\label{th2}\cite{MMM}
If $G$ is a connected graph with $\delta(G)\geq (L(G)+1)/2$, then $G$ is traceable.
\end{theo}

\begin{theo}\label{th3}\cite{MM2}
Let $G$ be a connected triangle-free graph with $L(G)\leq 2\delta(G)-1$. Then $G$ is either hamiltonian or $G\in \mathcal{F}_2$, where $\mathcal{F}_2$ is the class of non-hamiltonian graphs with leaf number $2\delta(G)-1$.
\end{theo}

Let $p(G)$ and $c(G)$ be the order of a longest path and a longest cycle in a graph $G$, respectively. Note that $c(G)$ is equal to the circumference of a graph $G$. Many researchers have investigated the relation between $p(G)$ and $c(G)$ (\cite{EHKS},\cite{L},\cite{OY},\cite{PY}). Motivated by Theorem \ref{th2}, we obtain the following main result.

\begin{theo}\label{th4}
Let $G$ be a connected graph of order $n$. If $L(G)\leq 2\delta(G) -1$, then $c(G)\geq n-1$. The bound is sharp and the condition cannot be relaxed.
\end{theo}

We also consider regular graphs.

\section{Main results}
We start with some lemmas that will be used repeatedly.
\begin{lem}\label{lem5}\cite{M32013}
Let $G$ be a connected graph of order $n$. If $L(G)\leq 2\delta(G)-1$, then $n\leq \max\{2\delta(G)+6, 3\delta(G)\}$.
\end{lem}

\begin{lem}\label{lem6}\cite{M32013}
Let $G$ be a connected graph with $L(G)\leq 2\delta(G)-1$. Then $G$ is 2-connected.
\end{lem}

For a graph $G$, $\kappa(G)$ and $\alpha(G)$ denote the connectivity and independent number of $G$, respectively. Let $\sigma_k(G)$ be the minimum degree sum of $k$ independent vertices of $G$ if $\alpha(G)\geq k$. $K_n$ stands for the \emph{complete graph} of order $n$.

\begin{lem}\label{lem7}\cite{CE}
Let $G$ be a connected graph. If $\kappa(G)\geq \alpha(G)$, then $G$ is hamiltonian except for $G=K_2$.
\end{lem}

\begin{lem}\label{lem8}\cite{EHKS}
Let $G$ be a 2-connected graph of order $n$. If $\sigma_3(G)\geq n+2$, then $c(G)\geq p(G)-1$.
\end{lem}

Now we first show that the result of Theorem \ref{th4} is true when $n\leq 3\delta(G)$.

\begin{lem}\label{lem9}\cite{M22013}
Let $G$ be a connected graph of order $n$. If $\delta(G)=2$ and $L(G)\leq 3$, then $c(G)\geq n-1$.
\end{lem}

Given graphs $G$ and $H$, the notation $G + H$ means the \emph{disjoint union} of $G$ and $H$. Then $tG$ denotes the disjoint union of $t$ copies of $G$. The notation $G \vee H$ means the \emph{joint} of $G$ and $H$. For graphs we will use equality up to isomorphism, so $G = H$ means that $G$ and $H$ are isomorphic.

For any graph $G$, $G[S]$ denotes the subgraph of $G$ induced by $S \subseteq V(G)$. Let $A,B\subseteq V(G)$ and $A\cap B=\emptyset$. Denote by $E(A,B)$ the set of edges of $G$ with one end in $A$ and the other end in $B$ and $e(A,B)=|E(A,B)|$.

\begin{lem}\label{lem10}\cite{EHKS}
Let $G$ be a connected graph with order $n\geq 3$. If $\sigma_3(G)\geq n$, then $G$ satisfies $c(G)\geq p(G)-1$ or $G\in\mathcal{F}(n)$, where $\mathcal{F}(n)$ is the class of graphs defined below.
\end{lem}

$\mathcal{F}(n)$ consists six subclasses:\\
$$\mathcal{F}(n)=\mathcal{F}_1(n)\cup\mathcal{F}_2(n)\cup\mathcal{F}_3(n)\cup\mathcal{F}_4(n)\cup\mathcal{F}_5(n)\cup\mathcal{F}_6(n).$$

For any graph $G\in \mathcal{F}(n)$, we have $|V(G)|=n$ and $\sigma_3(G)\geq n$. The subclasses are defined as follows (more details can be found in \cite{EHKS}):

$\mathcal{F}_1(n)$: $G\in \mathcal{F}_1(n)$ if $V(G)=A\cup B$ with $A\cap B=\emptyset$, $G[A]$ and $G[B]$ are hamiltonian or isomorphic to $K_2$, and $e(A,B)=1$.

$\mathcal{F}_2(n)$: $G\in \mathcal{F}_2(n)$ if $V(G)=A\cup B$ with $A\cap B=\{x\}$, $G[A]$ and $G[B]$ are both hamiltonian or both isomorphic to $K_2$, and $e(A\setminus\{x\},B\setminus\{x\})=0$.

$\mathcal{F}_3(n)$: $G\in \mathcal{F}_3(n)$ if $G$ is a 2-connected spanning subgraph of $K_2 \vee (K_a+ K_b+ K_c)$ with $a,b,c\geq 2$ ($n=a+b+c+2$).

$\mathcal{F}_4(n)$: $G\in \mathcal{F}_4(n)$ if $G$ is a 2-connected spanning subgraph of $K_3 \vee (aK_2+ bK_3)$ with $a,b\geq 0$ and $a+b=4$ ($n=2a+3b+3$, $11\leq n\leq 15$).

$\mathcal{F}_5(n)$: $G\in \mathcal{F}_5(n)$ if $G$ is a 2-connected spanning subgraph of $K_s \vee (sK_2+K_3)$ with $s\geq 4$ ($n=3s+3$).

$\mathcal{F}_6(n)$: $G\in \mathcal{F}_6(n)$ if $G$ is a 2-connected spanning subgraph of $K_s \vee (s+1)K_2$ with $s\geq 4$ ($n=3s+2$).

\begin{theo}\label{th11}
Let $G$ be a connected graph with order $n\leq 3\delta(G)$. If $L(G)\leq 2\delta(G)-1$, then $c(G)\geq n-1$.
\end{theo}
\begin{proof}
Let $G$ be a connected graph with order $n\leq 3\delta(G)$ and $L(G)\leq 2\delta(G)-1$. By Lemma \ref{lem6}, $G$ is 2-connected. If $\alpha(G)\leq 2$, by Lemma \ref{lem7}, then $G$ is hamiltonian and hence $c(G)\geq n-1$. Clearly, $\delta(G)\neq 1$. By Lemma \ref{lem9}, the result holds true for $\delta(G)= 2$. Now, it suffices to consider the case of $\alpha(G)\geq 3$ and $\delta(G)=\delta
\geq 3$. Note that $G$ is a connected graph with order $n>3$, by Theorem \ref{th2} and Lemma \ref{lem10}, $c(G)\geq p(G)-1=n-1$ or $G\in\mathcal{F}(n)$. Suppose to the contrary that $G\in \mathcal{F}(n)$.

Recall that $G$ is 2-connected. This implies that $G\notin \mathcal{F}_1(n)\cup \mathcal{F}_2(n)$. First suppose $G\in\mathcal{F}_3(n)$. For any vertex $x$ of $V(K_a)$, $d_{K_a}(x)\geq \delta-2$ in $G$ and hence $|V(K_a)|\geq \delta-1$. Similarly, $|V(K_b)|\geq \delta-1$ and $|V(K_c)|\geq \delta-1$. Then
$$|V(K_a)|+|V(K_b)|+|V(K_c)|+2\geq 3(\delta-1)+2=3\delta -1.$$
It implies that either $n=3\delta-1$ or $n=3 \delta$. For the first case, $|V(K_a)|=|V(K_b)|=|V(K_c)|=\delta-1$ and hence $G=K_2 \vee (K_a+K_b+K_c)$. It can easily be shown that $G$ contains a spanning tree with leaf number at least $2\delta$, a contradiction. For the second case, exactly one of $|V(K_a)|$, $|V(K_a)|$ and $|V(K_a)|$ is equal to $\delta$, and the rest are equal to $\delta-1$. Without loss of generality, suppose that $|V(K_a)|=\delta$ and $|V(K_b)|=|V(K_c)|=\delta-1$. Then $G[G-V(K_a)]=K_2 \vee (K_b+ K_c)$. The subgraph induced by the vertex set of $G[G-V(K_a)]$ with one vertex of $V(K_a)$ has a spanning tree with leaf number $2\delta$, contradicting $L(G)\leq 2\delta-1$. Thus $G\notin \mathcal{F}_3(n)$.

Next assume that $G\in \mathcal{F}_4(n)$. For $n\leq 3\delta-2$, by Lemmas \ref{lem6} and \ref{lem8}, $c(G)\geq p(G)-1$ since $n\leq \sigma_3(G)-2$. For $3\delta-1\leq n\leq 3\delta$, $\delta=4$ or $5$ since $11\leq n\leq 15$. Note that $a+b=4$ and $n=a+b+c+2$. If $\delta=4$, $n=11,a=4,b=0$ or $n=12,a=3,b=1$. It is easy to check that $L(G)\geq 8>2\delta-1$ in both cases. If $\delta=5$, $n=14,a=1,b=3$ or $n=15,a=0,b=4$. Since $\delta=5$, the first case is not allowed. For $n=15,a=0,b=4$, we have $L(G)\geq 10>2\delta-1$. Thus $G\notin \mathcal{F}_4(n)$.

Now assume that $G\in \mathcal{F}_5(n)$. For any vertex $x$ of $V(sK_2)$, $d_{K_s}(x)\geq \delta-1$ in $G$ and hence $|V(K_s)|\geq \delta-1$. Then $n=3s+3\geq 3(\delta-1)+3=3\delta$. Since $n\leq 3\delta$, then $s=\delta-1$. It implies that $G[V(K_s)\cup V(sK_2)]$ contains $(\delta-1)K_1\vee (\delta-1)K_2$ as a subgraph. Note that the subgraph induced by $V(K_3)$ contains no isolated vertex in $G$. Then we can split this into two cases. For $G[V(K_3)]=K_3$, it is easy to check that $G$ contains a spanning tree with leaf number at least $2\delta$, a contradiction. For $G[V(K_3)]=P_3$, let $w_1,w_2\in V(G[V(K_3)])$ and $d_{G[V(K_3)]}(w_1)=d_{G[V(K_3)]}(w_2)=1$. Then $d_{K_s}(w_1)= d_{K_s}(w_2)=\delta -1$. We also obtain $G$ has a spanning tree with leaf number at least $2\delta$, contradicting $L(G)\leq 2\delta-1$. Thus $G\notin \mathcal{F}_5(n)$.

It follows that $G\in \mathcal{F}_6(n)$. Since $n=3s+2\leq 3\delta$, we have $s\leq \delta-1$. For any vertex $x$ of $V((s+1)K_2)$, $d_{K_s}(x)\geq \delta-1$ in $G$ since $d(x)\geq \delta$. Then $s=\delta-1$ and $n=3s+2=3\delta-1$. Further, $G$ contains $K_1\vee \delta K_2$ as a subgraph. Thus, $G$ has a spanning tree with leaf number at least $2\delta$, a contradiction. This completes the proof of Theorem \ref{th11}.
\end{proof}

Before giving the proof of the main theorem, we prove a conclusion about regular graphs.

\begin{lem}\label{lem12}\cite{H}
Every 2-connected $k$-regular ($k\geq 3$) graph of order at most $3k + 3$ is hamiltonian except the Petersen graph $P$ and the graph obtained from $P$ by replacing one vertex of $P$ by a triangle.
\end{lem}
Denote by $P^{\triangle}$ be the graph obtained from $P$ by replacing one vertex of $P$ by a triangle.

\begin{figure}[h]
  \centering
  \subfigure
  {
      \begin{minipage}[b]{.4\linewidth}
      \centering
      \includegraphics[scale=0.6]{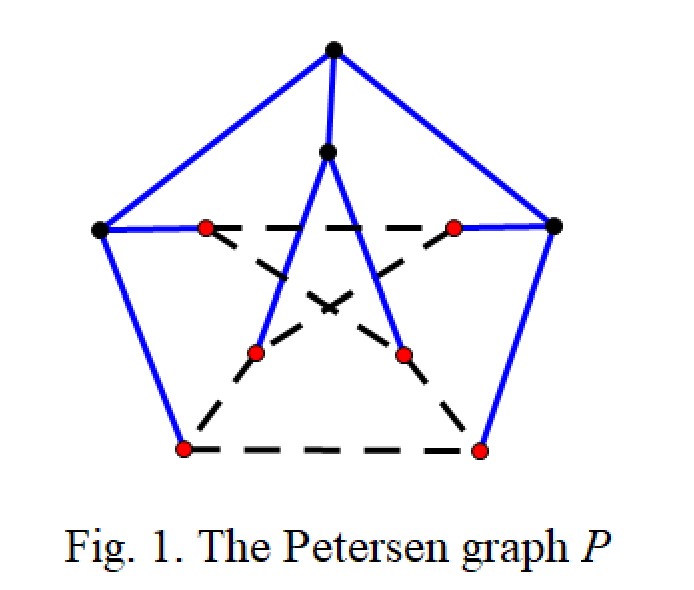}
      \end{minipage}
  }
  \subfigure
  {
      \begin{minipage}[b]{.4\linewidth}
      \centering
      \includegraphics[scale=0.6]{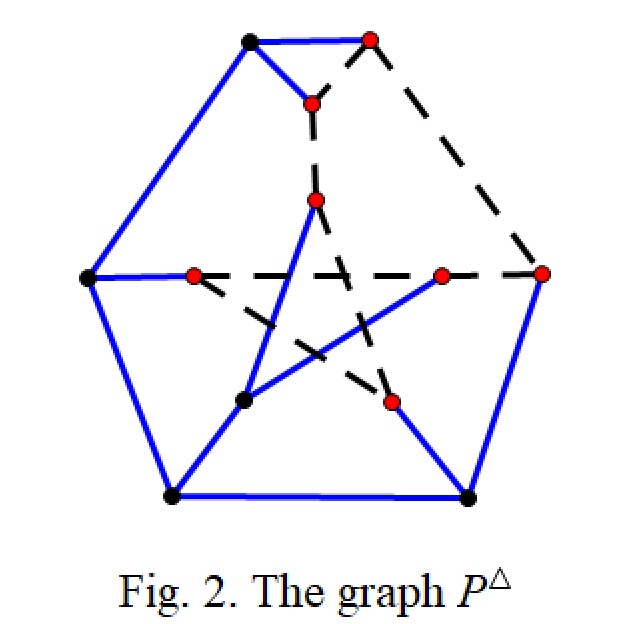}
      \end{minipage}
  }
\end{figure}

\begin{theo}\label{th13}
Let $G$ be a $k$-regular connected graph. If $L(G)\leq 2k-1$, then $G$ is hamiltonian and the condition cannot be relaxed.
\end{theo}
\begin{proof}
It is easy to verify that $L(P)= 6$ and $L(P^{\triangle})= 7$ (see Fig 1 and Fig 2). The Petersen graph $P$ is non-hamiltonian but satisfies $L(P)=6=2k$, so the condition cannot be relaxed.

Let $G$ be a $k$-regular connected graph of order $n$ with $L(G)\leq 2k-1$. Since $L(G)\geq 2$, then $k\geq 2$. Clearly, $G$ is hamiltonian when $k=2$. Next suppose that $k\geq 3$. By Lemma \ref{lem5}, we have $n\leq \max \{2k+6, 3k\}$. Note that $3k+3\geq \max \{2k+6, 3k\}$ when $k\geq 3$. By Lemmas \ref{lem6} and \ref{lem12}, $G$ is hamiltonian. This completes the proof of Theorem \ref{th13}.
\end{proof}

The following lemmas play the key role in the proof of Theorem \ref{th4}.

\begin{lem}\label{lem14}
Let $G$ be a connected graph with $L(G)\leq 2\delta(G)-1$. For $\delta(G)\geq 3$, if there is one vertex $x\in V(G)$ with degree $2\delta(G)-1$, then $|V(G)\setminus N[x]|\leq 2$.
\end{lem}
\begin{proof}
Let $G$ be a connected graph with $L(G)\leq 2\delta(G)-1$. Since $L(G)\leq 2\delta(G)-1$, each vertex of $N(x)$ has at most one neighbour in $V(G)\setminus N[x]$. By Lemma \ref{lem6}, $G$ is 2-connected. Then there are two vertices $y_1, y_2 \in V(G)\setminus N[x]$ have neighbors in $N(x)$. Similarly, each vertex of $\{y_1,y_2\}$ has at most one neighbour in $V(G)\setminus N[x]$ and hence at least $\delta(G)-1$ neighbors in $N(x)$. Suppose that $|V(G)\setminus N[x]|\geq 3$. There exists one vertex $y_3 \in V(G)\setminus (N[x]\cup \{y_1,y_2\})$ and $y_3$ is adjacent to $y_1$ or $y_2$. Clearly, $N(y_3)\cap N(x)=\emptyset$. Without loss of generality, assume that $y_3$ is adjacent to $y_1$. $G[N[x]\cup \{y_1\}\cup N(y_3)]$ contains a tree with leaf number $3\delta(G)-3$. Further, since $\delta(G)\geq 3$, we have $3\delta(G)-3>2\delta(G)-1$, a contradiction. This completes the proof of Lemma \ref{lem14}.
\end{proof}

\begin{lem}\label{lem15}\cite{D}
Let $G$ be a 2-connected graph of order $n$ and let $C$ be a longest cycle in $G$. Then $|V(C)|\geq \min\{n,2\delta(G)\}$.
\end{lem}

\begin{lem}\label{lem16}\cite{GW}
Let $G$ be a connected graph of order $n$.\\
(1) If $\delta(G)\geq 4$, then $L(G)\geq \frac{2n+8}{5}$.\\
(2) If $\delta(G)\geq 5$, then $L(G)\geq \frac{n}{2}+2$.
\end{lem}

\begin{lem}\label{lem17}
Let $G$ be a connected graph of order $n$ and let $C=c_1,c_2,\ldots,c_k,c_1$ be a longest cycle in $G$. The subscripts of the vertices $c_t$ are taken modulo $k$.\\
(1) The vertices $c_i$ and $c_{i+1}$ have no common neighbor in $V(G)\setminus V(C)$.\\
(2) Let $x,y\in V(G)\setminus V(C)$. If $c_i,c_j\in N_C(x)$, then $c_{i+1}$ and $c_{j+1}$ cannot both belong to $N(y)$.\\
(3) Let $P_C=p_1,p_2,\ldots,p_s$ be a longest path in $G-V(C)$. If the vertices $p_1$ and $p_s$ have distinct neighbors in $V(C)$, then $s\leq \lfloor\frac{k}{2}\rfloor-1$.
\end{lem}
\begin{proof}
It is easy to show that the results of Lemma \ref{lem17}, so we omit them.
\end{proof}
Finally, we show that the proof of Theorem \ref{th4}.
\begin{proof}
Let $G$ be a connected graph of order $n$. For $n\leq 3\delta(G)$, by Theorem \ref{th11}, $c(G)\geq n-1$. For $n\geq 3\delta(G)+1$, by Lemma \ref{lem5}, we have $\delta(G)\leq 5$. Clearly, $\delta(G)\neq 1$. By Lemma \ref{lem9}, the result is true when $\delta(G)= 2$. Denote by $\delta(G)=\delta$. Now, it suffices to consider the case of $3\delta+1\leq n\leq 2\delta+6$ and $3\leq\delta\leq 5$.

{\it Case 1.} $\delta=3$

Note that $10\leq n\leq 12$. Since $L(G)\leq 2\delta-1=5$, we have $3\leq\Delta(G)\leq 5$. If $\Delta(G)=3$, by Theorem \ref{th13}, $G$ is hamiltonian and hence $c(G)\geq n-1$. If $\Delta(G)=5$, by Lemma \ref{lem14}, $n\leq 6+2=8<10$, a contradiction. Next suppose that $\Delta(G)=4$. We discuss it in two Subcases according to the order of $G$.

{\it Subcase 1.1.} Consider $n=10$. Let $C=c_1,c_2,\ldots,c_k,c_1$ be a longest cycle in $G$ and let $P_C$ be a longest path in $G-V(C)$. By Lemmas \ref{lem6} and \ref{lem15}, $k\geq 6$. Now we show that $k\geq 9$. Suppose to the contrary that $6\leq k\leq 8$.

For $k=6$, by Lemma \ref{lem17} (3), $|V(P_C)|\leq 2$. Recall that $\delta=3$ and $\Delta(G)=4$, by Lemma \ref{lem17} (1) and (2), we obtain at most two isolated vertices in $G-V(C)$. Hence $|V(P_C)|=2$ and $G[V(G)\setminus V(C)]=2K_1+K_2$ or $2K_2$. Let $x,y\in V(G)\setminus V(C)$ and $x$ is adjacent to $y$. Since $k=6$, we have $d(x)=d(y)=3$ and $N(x)\setminus \{y\}=N(y)\setminus \{x\}$. Clearly, $N(x)\setminus \{y\}=\{c_1,c_4\}$ or $\{c_2,c_5\}$ or $\{c_3,c_6\}$. It is not difficult to see that the proof methods for the above three cases are similar. So let us just consider the first case. Note that $|V(G)\setminus V(C)|=4$ and $\Delta(G)=4$. There is one vertex $z\in V(G)\setminus V(C)$ is adjacent to at least one of $\{c_2,c_3,c_5,c_6\}$. Suppose $z$ is adjacent to $c_2$. The subgraph induced by $\{c_6,c_1,x,y,c_2,c_3,z\}$ contains a tree with leaf number 5. By Lemma \ref{lem14}, $n\leq 7+2=9$, a contradiction. The remaining cases can be proved in the same way.

For $k=7$, by Lemma \ref{lem17} (3), $|V(P_C)|\leq 2$. If $|V(P_C)|= 2$, then $G[V(G)\setminus V(C)]=K_1+K_2$. Let $x,y,z\in V(G)\setminus V(C)$ and $x$ is adjacent to $y$. It is easy to check that $d(x)=d(y)=3$ and $N(x)\setminus \{y\}=N(y)\setminus \{x\}$. Without loss of generality, suppose that $N(x)\setminus \{y\}=\{c_1,c_4\}$. Recall that $\delta=3$ and $\Delta(G)=4$. Then $z$ is adjacent to at least three vertices in $V(C)\setminus \{c_1,c_4\}$. If $z$ is adjacent to $c_2$, the subgraph induced by $\{c_6,c_1,x,y,c_2,c_3,z\}$ contains a tree with leaf number 5. By Lemma \ref{lem14}, $n\leq 7+2=9$, a contradiction. Using a similar argument as above, we deduce that $z$ is not adjacent to $c_3$, $c_5$ and $c_6$, contradicting $d(z)\geq 3$. Next suppose $|V(P_C)|= 1$. Then $G[V(G)\setminus V(C)]=3K_1$. Let $x,y,z\in V(G)\setminus V(C)$. By Lemma \ref{lem17} (1) and (2), $d(x)=d(y)=d(z)=3$ and $N(x)=N(y)=N(z)$. Then there is one vertex of $N(x)$ has degree at least 5, a contradiction.

For $k=8$, $|V(P_C)|\leq 2$ since $n=10$. If $|V(P_C)|= 2$, then $G[V(G)\setminus V(C)]=K_2$. Let $x,y\in V(G)\setminus V(C)$. Without loss of generality, suppose $x$ is adjacent to $c_1$. Since $C$ is a longest cycle in $G$, then $y$ is adjacent to $c_4$, $c_5$ or $c_6$. Obviously, we can split this into two cases. The first case is where $y$ is adjacent to $c_4$. One can easily show that $d(x)=d(y)=3$ and $N(x)\setminus\{y\}=N(y)\setminus\{x\}=\{c_1,c_4\}$. Consider the vertex $c_2$. If $c_2$ is not adjacent to $c_8$, the subgraph induced by $N(c_1)\cup N(c_2)$ contains a tree with leaf number 5 since $d(c_2)\geq 3$. By Lemma \ref{lem14}, $n\leq 7+2=9<10$, a contradiction. If $c_2$ is adjacent to $c_8$, $G$ contains a cycle $c_2,c_8,c_7,c_6,c_5,c_4,y,x,c_1,c_2$ with length 9 (see Fig. 3), contradicting to $k=8$. The second case is where $y$ is adjacent to $c_5$. Similarly, we have $d(x)=d(y)=3$ and $N(x)\setminus\{y\}=N(y)\setminus\{x\}=\{c_1,c_5\}$. The following results which are derived from the above proof: $c_2$ is adjacent to $c_8$ and $c_4$ is adjacent to $c_6$. Then $G$ contains a cycle $c_2,c_8,c_7,c_6,c_4,c_5,y,x,c_1,c_2$ with length 9, a contradiction.
\begin{figure}[h]
  \centering
  \subfigure
  {
      \begin{minipage}[b]{.4\linewidth}
      \centering
      \includegraphics[scale=0.6]{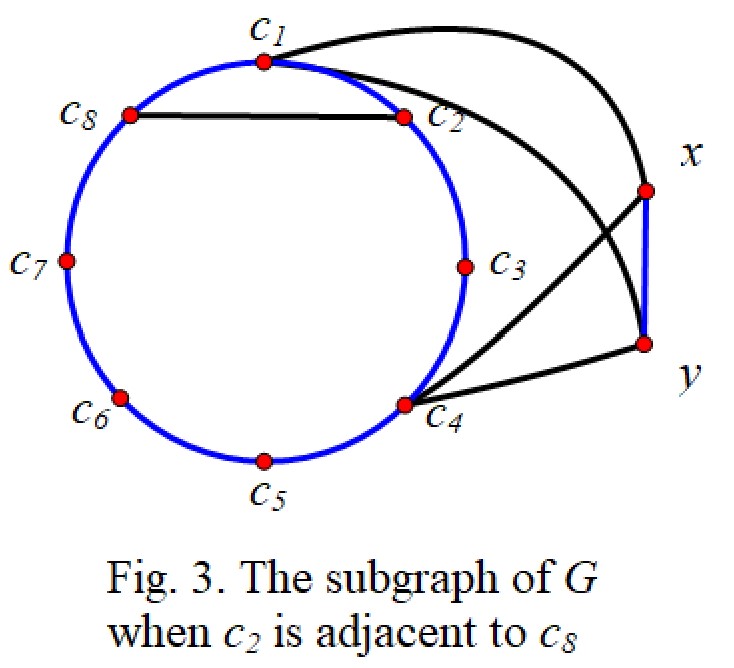}
      \end{minipage}
  }
  \subfigure
  {
      \begin{minipage}[b]{.4\linewidth}
      \centering
      \includegraphics[scale=0.6]{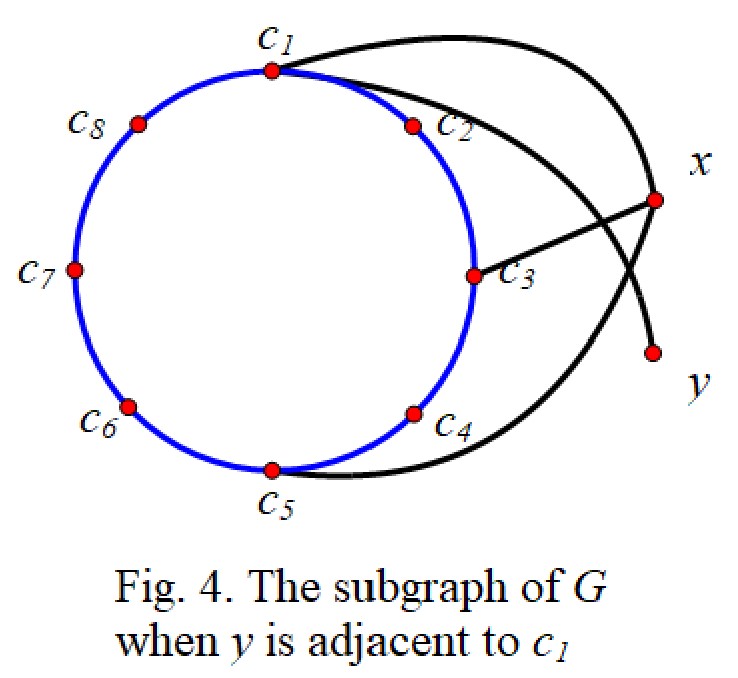}
      \end{minipage}
  }
\end{figure}
Next suppose $|V(P_C)|= 1$ and hence $G[V(G)\setminus V(C)]=2K_1$. Let $x,y\in V(G)\setminus V(C)$. Since $C$ is a longest cycle in $G$ and $L(G)\leq 5$, then $d(x)=d(y)=3$. Without loss of generality, suppose that $N(x)=\{c_1,c_3,c_5\}$ or $\{c_1,c_3,c_6\}$. For the first case, assert that $N(y)\cap N(x)=\emptyset$. Otherwise, $G$ contains a tree with leaf number at least 6 if $y$ is adjacent to $c_1$ or $c_5$ (see Fig.4), a contradiction. And if $y$ is adjacent to $c_3$, the subgraph induced by the vertex set $\{c_1,c_2,c_3,c_4,c_5,x,y\}$ contains a tree with leaf number 5. Then, by Lemma \ref{lem14}, $n\leq 9$, a contradiction. So, $N(y)\subseteq V(C)\setminus N(x)$ and $|N(y)|=3$, by Lemma \ref{lem17} (1) and (2), which is not allowed. For the second case, the proof method is similar to the first case, and will not be repeated here.

{\it Subcase 1.2.} Consider $n=11$ or $12$. Let $x\in V(G)$ with $d(x)=4$. Set $N(x)=\{x_1,x_2,x_3,x_4\}$. Assert that any vertex of $N(x)$ has at most one neighbor in $V(G)\setminus N[x]$. Since $L(G)\leq 5$, we have $d_{G- N[x]}(x_i)\leq 2$ for $i\in \{1,2,3,4\}$. If there is one vertex of $N(x)$ has exactly two neighbors in $V(G)\setminus N[x]$, by Lemma \ref{lem14}, $n\leq 5+2+2=9<11$, a contradiction. Hence, $e(N(x),V(G)\setminus N[x])\leq 4$. Let $N_2(x)\subseteq V(G)\setminus N[x]$ and each vertex of $N_2(x)$ has neighbor in $N(x)$. Similarly, by Lemma \ref{lem14}, we can show that each vertex of $N_2(x)$ has at most one neighbour in $V(G)\setminus N[x]$. Then each vertex of $N_2(x)$ has at least two neighbours in $N(x)$, since $\delta=3$. By Lemma \ref{lem6}, $G$ is 2-connected and hence $|N_2(x)|\geq 2$. Then $|N_2(x)|= 2$ and $e(N(x),V(G)\setminus N[x])= 4$. Set $N_2(x)=\{y_1,y_2\}$. Without loss of generality, suppose that $N(y_1)\cap N(x)=\{x_1,x_2\}$ and $N(y_2)\cap N(x)=\{x_3,x_4\}$. It is easy to check that $y_1$ is not adjacent to $y_2$, since $n\geq 11$. Let $z_1=N(y_1)\setminus\{x_1,x_2\}$ and $z_2= N(y_2)\setminus\{x_3,x_4\}$. Since $G$ is 2-connected, then $z_1\neq z_2$. Note that $G[N(x)]$ contains $2K_2$. Then $G$ contains a path of length 8 with endpoints $z_1$ and $z_2$. For $n=11$, it remains two vertices $w_1$ and $w_2$. Obviously, $d(w_1)=d(w_2)=3$ and $N(w_1)=\{w_2,z_1,z_2\}$, $N(w_2)=\{w_1,z_1,z_2\}$. Thus $c(G)=n$. For $n=12$, it remains three vertices $w_1$, $w_2$ and $w_3$. One can easy show that $c(G)\geq n-1$. So Case 1 is proven.

{\it Case 2.} $\delta=4$

Note that $13\leq n\leq 14$. For $n=14$, by Lemma \ref{lem16} (1), $L(G)\geq \frac{2n+8}{5}=\frac{36}{5}>7$, contradicting $L(G)\leq 2\delta-1=7$. Then we only need to consider $n=13$. Suppose $n=13$. Since $L(G)\leq 7$, we have $4 \leq \Delta(G) \leq 7$. By Lemma \ref{lem14}, $\Delta(G)\neq 7$. For $\Delta(G)=6$, let $x\in V(G)$ with $d(x)=6$. Then any vertex of $N(x)$ has at most two neighbors in $V(G)\setminus N[x]$. Let $N_2(x)\subseteq V(G)\setminus N[x]$ and each vertex of $N_2(x)$ has neighbor in $N(x)$. By Lemma \ref{lem14}, any vertex of $N(x)\cup N_2(x)$ has at most one neighbor in $V(G)\setminus N[x]$. This implies that $e(N(x),N_2(x))\leq 6$. Since $\delta=4$, then $|N_2(x)|\leq 2$. By Lemma \ref{lem6}, $G$ is 2-connected and hence $|N_2(x)|=2$. Clearly, $y_1$ is not adjacent to $y_2$ since $n=13$. Let $z_1\in N(y_1)$ and $z_2\in N(y_2)$. Recall that $G$ is 2-connected. $z_1\neq z_2$. Hence $|N(z_1)\setminus \{y_1\}|\geq 3$. From Fig.5, we obtain $G$ contains a tree with leaf number 8, a contradiction. For $\Delta(G)=4$, by Theorem \ref{th13}, $c(G)=n\geq n-1$.
\begin{figure}[h]
  \centering
      \includegraphics[scale=0.6]{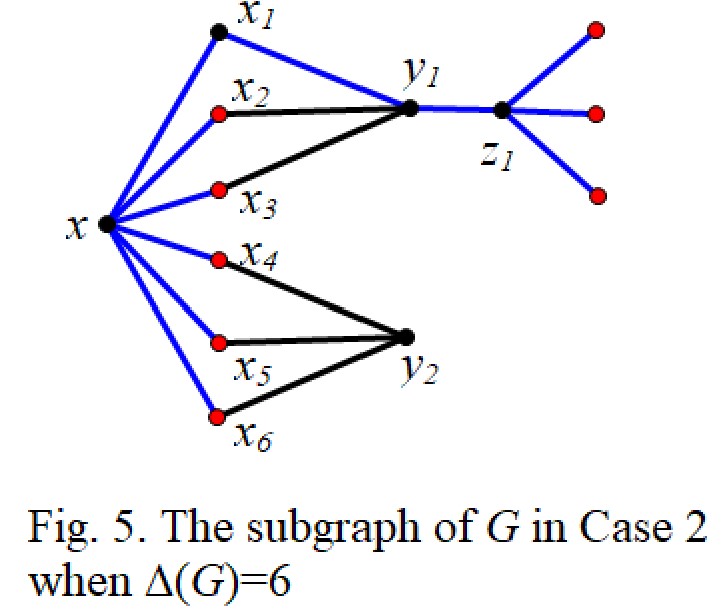}
\end{figure}

It remains the case of $\Delta(G)=5$. Let $C=c_1,c_2,\ldots,c_k,c_1$ be a longest cycle in $G$ and let $P_C$ be a longest path in $G-V(C)$. By Lemmas \ref{lem6} and \ref{lem15}, $k\geq 8$. Now we show that $k\geq 12$. Suppose to the contrary that $8\leq k\leq 11$.

For $k=8$, by Lemma \ref{lem17} (3), $|V(P_C)|\leq 3$. Since $\Delta(G)=5$, by Lemma \ref{lem17} (1) and (2), there are at most three isolated vertices in $G-V(C)$. Then $2\leq |V(P_C)|\leq 3$. If $|V(P_C)|=2$, then $G-V(C)=3K_1+K_2$ or $K_1+2K_2$. Let $x,y\in V(G)\setminus V(C)$ and $x$ is adjacent to $y$. Without loss of generality, suppose that $c_1\in N(x)$. By Lemma \ref{lem17} (1) and (2), $N(y)\subseteq\{x,c_1,c_4,c_5,c_6\}$. Then $N(y)=\{x,c_1,c_4,c_6\}$, since $C$ is a longest cycle and $\delta=4$. Further, we have $N(x)=\{c_1,y\}$, contradicting $d(x)\geq 4$. Next suppose $|V(P_C)|=3$. Let $P_C=x,y,z$. It follows that $N_C(x)=N_C(z)$ and $|N_C(x)|=|N_C(z)|=2$. Recall that $C$ is a longest cycle in $G$. Without loss of generality, suppose that $N_C(x)=N_C(z)=\{c_1,c_5\}$. Since $\delta(G)=4$ and $|V(P_C)|=3$, we have $x$ is adjacent to $z$. Note that we have a new path $P'_C=y,x,z$ in $G-V(C)$. Similarly, $N_C(y)=N_C(z)$ and $|N_C(y)|=|N_C(z)|=2$. So, $N_C(x)=N_C(y)=N_C(z)=\{c_1,c_5\}$ and $d(x)=d(y)=d(z)=4$. Then $G[V(G)\setminus V(C)]=2K_1+K_3$. Let $\{u,v\}=V(G)\setminus (V(C)\cup\{x,y,z\})$. Clearly, $N(u)=N(v)=\{c_2,c_4,c_6,c_8\}$. The subgraph induced by $\{c_8,c_1,c_2,c_3,x,y,z,u,v\}$ contains a tree with leaf number 7 (see Fig.6). By Lemma \ref{lem14}, $n\leq 9+2=11<13$, a contradiction.
\begin{figure}[h]
  \centering
      \includegraphics[scale=0.6]{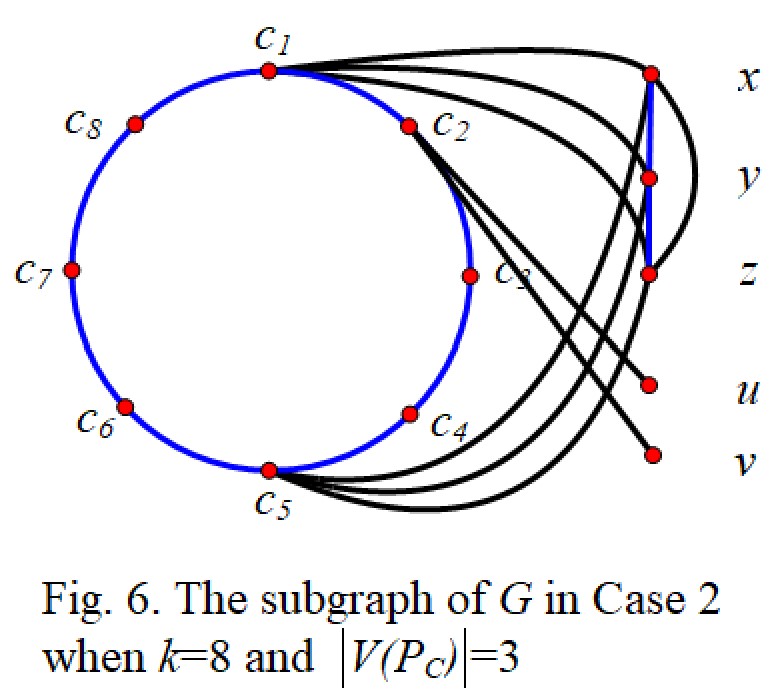}
\end{figure}

For $k=9$, by Lemma \ref{lem17}, $|V(P_C)|\leq 3$ and there are at most two isolated vertices in $G-V(C)$. Then $2\leq |V(P_C)|\leq 3$. If $|V(P_C)|=2$, then $G[V(G)\setminus V(C)]=2K_1+K_2$ or $2K_2$. Let $x,y\in V(G)\setminus V(C)$ and $x$ is adjacent to $y$. One can easy show that $d(x)=d(y)=4$ and $N(x)\cap N(y)=\{c_1,c_4,c_7\}$ or $\{c_2,c_5,c_8\}$ or $\{c_3,c_6,c_9\}$. Without loss of generality, suppose that $N(x)\cap N(y)=\{c_1,c_4,c_7\}$. Since $L(G)\leq 7$, any vertex of $V(G)\setminus (V(C)\cup \{x,y\})$ has no neighbor in $\{c_1,c_4,c_7\}$. Further, by Lemma \ref{lem17} (1) and (2), $G[V(G)\setminus V(C)]\neq 2K_1+K_2$ and $G[V(G)\setminus V(C)]\neq 2K_2$, since $\delta=4$. Next suppose $|V(P_C)|=3$. Let $P_C=x,y,z$. Using the same method as the case of $k=8$ and $|V(P_C)|=3$, we obtain $G[V(G)\setminus V(C)]=K_1+K_3$ and $N_C(x)=N_C(y)=N_C(z)$. Without loss of generality, suppose that$N_C(x)=N_C(y)=N_C(z)=\{c_1,c_5\}$. Consider the vertex $c_2$. If $c_2$ is adjacent to $c_9$, $G$ contains a cycle $c_1,c_2,c_9,c_8,c_7,c_6,c_5,z,y,x,c_1$ with length 10, a contradiction. If $c_2$ is not adjacent to $c_9$, the subgraph induced by $N(c_1)\cup N(c_2)$ contains a tree with leaf number 7. By Lemma \ref{lem14}, $n\leq 10+2=12<13$, a contradiction.

For $k=10$, $|V(P_C)|\leq 3$. For $|V(P_C)|= 3$, let $V(G)\setminus V(C)=\{x,y,z\}$. Similarly, one can easy show that $G[V(G)\setminus V(C)]=K_3$ and $N_C(x)=N_C(y)=N_C(z)$ and $|N_C(x)|=|N_C(y)|=|N_C(z)|=2$. Without loss of generality, suppose that $N_C(x)=\{c_1,c_5\}$ or $\{c_1,c_6\}$. If $N_C(x)=\{c_1,c_5\}$, we consider the vertex $c_3$. Note that $d(c_1)=d(c_5)=5$. Then $|N(c_3)\cap (V(C)\setminus \{c_1,c_2,c_4,c_5\})|\geq 2$, since $\delta=4$. If $c_3$ is adjacent to $c_6$ (see Fig.7a) or $c_{10}$, then $G$ contains a cycle $c_5,c_4,c_3,c_6,c_7,c_8,c_9,c_{10},c_1,x,y,z,c_5$ with length 12 or a cycle $c_1,c_2,c_3,c_{10},c_9,c_8,c_7,c_6,c_5,z,y,x,c_1$ with length 12, a contradiction. If $c_3$ is adjacent to $c_7$ (see Fig.7b) or $c_9$, then $G$ contains a cycle $c_5,c_4,c_3,c_7,c_8,c_9,c_{10},c_1,x,y,z,c_5$ with length 11 or a cycle $c_1,c_2,c_3,c_9,c_8,c_7,c_6,c_5,z,y,x,c_1$ with length 11, a contradiction. Hence we have $N(c_3)\subseteq\{c_2,c_4,c_8\}$, contradicting $\delta=4$.
\begin{figure}[h]
  \centering
  \subfigure
  {
      \begin{minipage}[b]{.4\linewidth}
      \centering
      \includegraphics[scale=0.6]{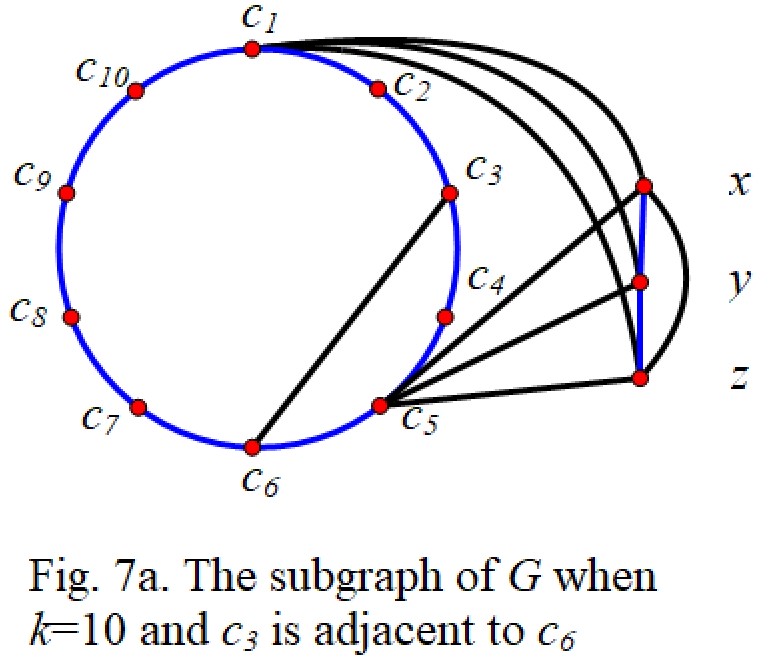}
      \end{minipage}
  }
  \subfigure
  {
      \begin{minipage}[b]{.4\linewidth}
      \centering
      \includegraphics[scale=0.6]{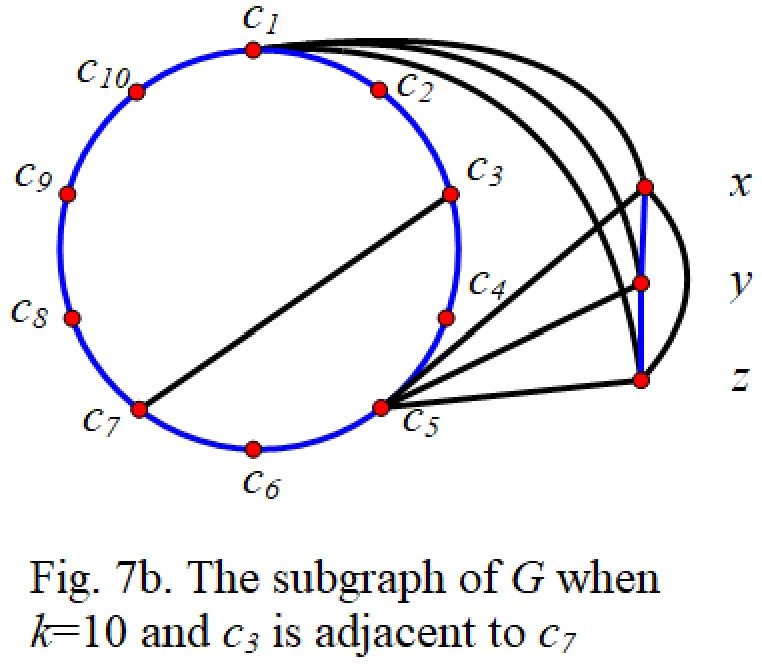}
      \end{minipage}
  }
\end{figure}
If $N_C(x)=\{c_1,c_6\}$, using the same method, we have $c_7,c_{10}\notin N(c_3)$ and $c_7,c_{10}\notin N(c_4)$. Now we show that $c_3$ and $c_4$ are both adjacent to exactly one of $\{c_8,c_9\}$. Suppose to the contrary that $c_3$ and $c_4$ are adjacent to $c_8$ and $c_9$ or $c_9$ and $c_8$, respectively. Then $G$ contains a cycle $c_1,c_2,c_3,c_8,c_9,c_4,c_5,c_6,z,y,x,c_1$ with length 11 or a cycle $c_1,c_2,c_3,c_9,c_8,c_4,c_5,c_6,z,y,x,c_1$ with length 11, contradicting $k=10$. Without loss of generality, suppose $c_9\in N(c_3)\cap N(c_4)$. Since $\delta=4$, $c_4$ is adjacent to $c_2$. Then $G$ contains a cycle $c_1,c_{10},c_9,c_3,c_2,c_4,c_5,c_6,z,y,x,c_1$ with length 11, a contradiction. For $|V(P_C)|\leq 2$, it implies that there exists at least one isolated vertex in $G-V(C)$. Assert that any isolated vertex of $G-V(C)$ has degree 4. Otherwise, suppose $d_{G-V(C)}(x)=0$ and $d(x)=5$. Then $N(x)=\{c_1,c_3,c_5,c_7,c_9\}$ or $\{c_2,c_4,c_6,c_8,c_{10}\}$. We show that the first case, the second can be proved by same method. Note the subgraph induced by $\{c_9,c_{10},c_1,c_2,c_3,c_4,c_5,c_6,c_7,x\}$ contains a tree with leaf number 7. By Lemma \ref{lem14}, $n\leq 10+2=12<13$, a contradiction. Let $V(G)\setminus V(C)=\{x,y,z\}$ and $d_{G-V(C)}(x)=0$. Then, by Lemma \ref{lem14}, $N(y)\cap N(x)=\emptyset$ and $N(z)\cap N(x)=\emptyset$. By Lemma \ref{lem17} (1) and (2), $|V(P_C)|>1$ since $\delta=4$. Hence $|V(P_C)|=2$ and $y$ is adjacent to $z$. Since $d(x)=4$, then the neighbors of $x$ divide $C$ into four parts. The lengths of the four parts of $C$ are 2,2,2,4 or 2,2,3,3 or 2,3,2,3 (see Fig.8). By Lemma \ref{lem17} (1) and (2), it is easy to show that in each case there is a contradiction, so we omit it.
\begin{figure}[h]
  \centering
      \includegraphics[scale=0.5]{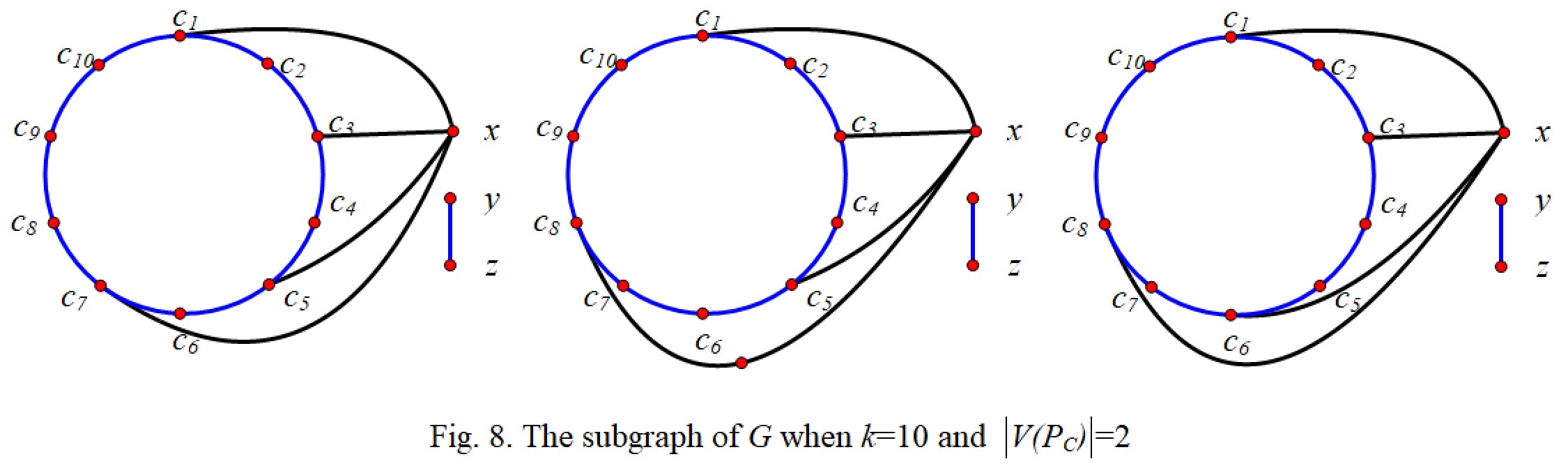}
\end{figure}

For $k=11$, $|V(P_C)|\leq 2$. Let $x,y\in V(G)\setminus V(P)$. Since $n=13$, by Lemma \ref{lem14}, $d(x)=d(y)=4$. For $|V(P_C)|=2$, $G[V(G)\setminus V(C)]=K_2$. Note that the neighbors of $x$ in $V(C)$ divide $C$ into three parts. Then we have five cases and the lengths of the three parts of $C$ are 2,2,7 or 2,3,6 or 2,4,5 or 3,3,5 or 3,4,4. The proof methods for the first three cases are similar. Since $C$ is a longest cycle in $G$, one can easy show that $d(y)\leq 3<4$, a contradiction. The proofs for the latter two cases are similar, so we only give the proof for one of them here. Without loss of generality, suppose $N(x)=\{c_1,c_4,c_7\}$. Consider the vertex $c_3$. We assert that $N(c_3)\cap N(x)=\{c_4\}$. Otherwise, if $c_3$ is adjacent to $c_1$, the subgraph induced by $\{c_{11},c_1,c_2,c_3,x,y,c_4,c_6,c_7,c_8\}$ contains a tree with leaf number 7 and hence by Lemma \ref{lem14}, $n\leq 12<13$, a contradiction. Similarly, we can show that $c_3$ is not adjacent to $c_7$. If $c_3$ is adjacent to $c_5$, $G$ contains a cycle $c_4,c_3,c_5,c_6,c_7,c_8,c_9,c_{10},c_{11},c_1,x,y,c_4$ of length 12, contradicting $k=11$. Hence $N(c_3)\cap \{x,y,c_1,c_5,c_7\}=\emptyset$. Note that the subgraph induced by $N[c_3]\cup \{x,y,c_1,c_5,c_7\}$ contains a tree with leaf number 7, since $d(c_3)\geq 4$. Then, by Lemma \ref{lem14}, $n\leq 12<13$, a contradiction. For $|V(P_C)|=1$, $G[V(G)\setminus V(C)]=2K_1$. Recall that $d(x)=d(y)=4$. We assert that $N(x)\cap N(y)=\emptyset$. Otherwise, by Lemma \ref{lem14}, we will have $n<13$. Note that the neighbors of $x$ in $V(C)$ divide $C$ into four parts and hence we have four cases. The lengths of each parts of $C$ are 2,2,2,5 or 2,2,3,4 or 2,3,2,4 or 2,3,3,3. Similarly, by Lemma \ref{lem17} (1) and (2), one can easy show that in each case there is a contradiction. Then Case 2 is proven.

{\it Case 3.} $\delta=5$

Note that $n=16$. By Lemma \ref{lem16} (2), $L(G)\geq \frac{n}{2}+2=10$, contradicting $L(G)\leq 2\delta-1$. So Case 3 is proven.

For the sharpness, consider the following graph. The graph $G_1$ of order $n$ is formed by taking the cycle $C_{n-1}=v_1,v_2,\ldots,v_{n-2},v_{n-1},v_1$ and add one vertex $v_{n}$ together with edges $v_1v_n, v_3v_n$. Note that $\delta(G_1)=2$ and $L(G_1)=3$. Then $G_1$ satisfying $L(G_1)\leq 2\delta(G_1)-1$ and $c(G_1)=n-1$.

The condition $L(G)\leq 2\delta(G)-1$ cannot be relaxed. The graph $G_2$ with order $n\geq 8$ is formed by taking the cycle $C_{n-2}=v_1,v_2,\ldots,v_{n-2},v_1$ and add two vertices $v_{n-1}$ and $v_n$ together with edges $v_1v_{n-1}, v_3v_{n-1}, v_{n-5}v_n, v_{n-3}v_n$. Clearly, $\delta(G_2)=2$ and $L(G_2)=4$. Then $G_2$ satisfying $L(G_2)\leq 2\delta(G_2)$ but $c(G_2)=n-2$. This completes the proof of Theorem \ref{th4}.
\end{proof}

{\bf Acknowledgement} This research was supported Science and Technology Commission of Shanghai Municipality (STCSM) grant 18dz2271000.

\end{document}